\documentclass[a4paper,12pt]{article}

\usepackage{mathtools}

\usepackage{amstext,amssymb,amsmath,amsthm}
\usepackage{relsize,xspace,ifthen}
\usepackage[T1]{fontenc}
\usepackage{microtype}
\usepackage{enumerate}
\usepackage[utf8]{inputenc}
\usepackage{tikz}
\usetikzlibrary{arrows}
\usetikzlibrary{shapes.geometric}
\usetikzlibrary{calc, decorations.markings}
\usepackage{hyperref}
\usepackage{comment}

\usepackage{mathrsfs}
\usepackage[font=small]{caption}

\usepackage[T1]{fontenc}
    \usepackage{ae}
    \usepackage{aecompl}

\newenvironment{cenv}{\begin{list}{}{%
      \setlength{\labelwidth}{1.5em}%
      \setlength{\leftmargin}{\labelwidth}%
      \addtolength{\leftmargin}{\labelsep}%
      \setlength{\listparindent}{0em}%
      \setlength{\topsep}{10pt}%
      \setlength{\itemsep}{5pt}%
      \setlength{\parsep}{0pt}%
    }
  }{
  \end{list}
}

\newcounter{claimcounter}

\usepackage{latexsym}
\usepackage{nicefrac}

\usepackage{algorithm}
\usepackage{algpseudocode}
\usepackage{eqparbox}

\newtheorem{theorem}{Theorem}
\newtheorem{lemma}{Lemma}

\newtheorem{rem}{Remark}
\let\oldrem\rem
\renewcommand{\rem}{\oldrem\normalfont}

\numberwithin{equation}{section}
\numberwithin{theorem}{section}
\numberwithin{lemma}{section}

 \newcommand{\BBB}{\mathcal{B}}
\newcommand{\CCC}{\mathcal{C}} \newcommand{\DDD}{\mathcal{D}}
 \newcommand{\FFF}{\mathcal{F}}
 \newcommand{\HHH}{\mathcal{H}}
 
\newcommand{\KKK}{\mathcal{K}} 
\newcommand{\MMM}{\ensuremath{\mathcal{M}}} 
 \newcommand{\PPP}{\mathcal{P}}
 
 \newcommand{\TTT}{\mathcal{T}}
\newcommand{\UUU}{\mathcal{U}} 
 \newcommand{\XXX}{\mathcal{X}}
 \newcommand{\ZZZ}{\mathcal{Z}}

\renewcommand{\phi}{\varphi}
\renewcommand{\epsilon}{\varepsilon}

\newcommand{\ie}{i.e.\@\xspace}

\DeclareMathOperator{\tw}{tw}
\DeclareMathOperator{\mw}{mw}
\DeclareMathOperator{\wsmw}{ws-mw}
\DeclareMathOperator{\smw}{s-mw}

\DeclareMathOperator{\hw}{\mathcal{H}w}
\DeclareMathOperator{\kw}{\mathcal{K}w}

\title{On the Medianwidth of Graphs}
\author{Konstantinos Stavropoulos \\ \small{RWTH Aachen University} \\ \small{stavropoulos@informatik.rwth-aachen.de}}
\date{}

\begin{document}

\maketitle

\begin{abstract}
A \emph{median graph} is a connected graph, 
such that for any three vertices $u,v,w$ there is exactly one vertex $x$ that lies simultaneously
on a shortest $(u,v)$-path, a shortest $(v,w)$-path and a shortest $(w,u)$-path. Examples of median graphs are trees, grids and hypercubes.

We introduce and study a generalisation of tree decompositions, to be called \emph{median decompositions}, where instead of decomposing a graph $G$ in a treelike fashion, 
we use median graphs as the underlying graph of the decomposition. We show that the corresponding width parameter $\mw(G)$, the \emph{medianwidth} of $G$,
is equal to the \emph{clique number} of the graph, while a suitable variation of it is equal to the \emph{chromatic number} of $G$.

We study in detail the \emph{$i$-medianwidth} $\mw_i(G)$ of a graph, for which we restrict the underlying median graph of a decomposition 
to be isometrically embeddable to the Cartesian product of $i$ trees.
For $i\geq 1$, the parameters $\mw_i$ constitute a hierarchy starting from treewidth and converging to the clique number.
We characterize the $i$-medianwidth of a graph to be, roughly said, the largest ``intersection'' of the best choice of
$i$ many tree decompositions of the graph. 

Lastly, we extend the concept of tree and median decompositions and propose a general framework of how to decompose a graph in any fixed graphlike fashion.
\end{abstract}

\section{Introduction}\label{sec:intro}
The notion of \emph{tree decompositions} and \emph{treewidth} was first introduced (under different names) by Halin \cite{halin1976s}.
It also arose as a natural and powerful tool in the fundamental work of Robertson and Seymour on graph minors, who reintroduced it in its more standard in the literature 
form \cite{robertson1984graph,robertson1986graph}. 

Treewidth, denoted by $\tw(G)$, can be seen as a measure of how ``treelike'' a graph is and has turned out to, actually, be a connectivity measure of graphs (see \cite{reed1997tree}). 
The usefulness of tree decompositions as a decomposition tool, especially in the theory of Graph Minors, is highlighted by various, often very general, structural theorems 
(\cite{robertson1991graph,robertson2003graph,carmesin2013canonical,carmesin2015canonical,carmesin2014connectivity,hundertmark2011profiles}).
Moreover, various NP-hard decision and optimization problems are fixed-parameter tractable when parameterized by treewidth (see \cite{FlumGro06,Bodlaender93,Bodlaender05}).

The concept of modelling a graph like a ``thick'' tree has been fundamental for the ability to study graph classes excluding a fixed minor. 
A next step to take would be to study what happens beyond those classes. A very robust approach by Ne\v set\v ril and Ossona de Mendez to study graph classes beyond bounded treewidth 
(or even excluding a fixed minor)---but more in terms of sparsity rather than trying to model the whole graph after another graph like in the spirit of tree decompositions---can
be seen here \cite{BndExpI,BndExpII,NesetrilO11,NesetrilOdM12}.

Since graphs of bounded treewidth inherit several advantages of trees, it has been tempting to investigate how to go beyond tree decompositions 
and try to model a graph on graphs other than trees (in the sense that the former has ``bounded width'' in terms of the latter), maybe as a means to study how these more general decompositions 
can be used to form structural hierarchies of graph classes. For example, Diestel and K{\"u}hn proposed a version of such general decompositions with interesting implications in \cite{diestel2005graph}, 
who also note a disadvantage in their decompositions: all graphs, when modelled like a grid, have bounded ``gridwidth''. 

A median graph is a connected graph, 
such that for any three vertices $u,v,w$ there is exactly one vertex $x$ that lies simultaneously
on a shortest $(u,v)$-path, a shortest $(v,w)$-path and a shortest $(w,u)$-path. Examples of median graphs are grids and the $i$-dimensional hypercube $Q_i$, for every $i\geq 1$. 
One of the simplest examples of median graphs are trees themselves.

One might choose to see trees as the one-dimensional median graphs under a certain perspective: for example, the topological dimension of a tree continuum is one; or amalgamating one-dimensional cubes,
namely edges, on a tree, will also produce a tree; or trees are the median graphs not containing a square (the two-dimensional cube) as an induced subgraph \cite{mulder1980interval}.

A subset $S$ of vertices of a graph is \emph{(geodesically) convex} if for every pair of vertices in $S$, all shortest paths between them only contain vertices in $S$. 
The following is the core observation that inspired this paper:

\begin{center}
 Convexity degenerates to connectedness on trees!
\end{center}

In a tree decomposition, a vertex of the graph lives in a connected subgraph of the underlying tree.
The properties of convex subsets of median graphs, one of them being the \emph{Helly Property}, 
provide the means allowing the extension of the concept of tree decompositions into the setting of median decompositions
in a rather natural way: when we use general median graphs as the underlying graph of the decomposition, a vertex of the original graph will live in a convex subgraph.
This generalisation of tree decompositions will, as a result, allow for finer decompositions of the decomposed graph.

In Section \ref{sec:prelim}, we summarize some relevant parts of the known theory on median graphs.

In Section \ref{sec:meddec}, we introduce median decompositions and study their general properties, some of which are natural translations of corresponding properties of tree decompositions.
We also prove that the corresponding width parameter $\mw(G)$ matches the \emph{clique number} $\omega(G)$ of a graph $G$, the size of its largest complete subgraph. 

Section \ref{sec:chi} is devoted to a specific variation of median decompositions, which satisfy an additional axiom ensuring more regularity for them.
Certain median decompositions, which we will call \emph{chromatic median decompositions} and arise by making use of a proper colouring of the graph, enjoy this additional regularity by their definition.
This allows us to see that the respective width parameter, to be called \emph{smooth medianwidth}, is equivalent to the \emph{chromatic number} $\chi(G)$ of $G$. 

Every median graph can be isometrically embedded into the Cartesian product of a finite number of trees. In Section \ref{sec:i-mw}, 
we consider median decompositions whose underlying median graph must be isometrically embeddable into the Cartesian product of $i$ trees, 
along with the respective medianwidth parameter, to be called $i$-medianwidth $\mw_i(G)$. 
By definition, the invariants $\mw_i$ will form a non-increasing sequence:
$$\tw(G)+1=\mw_1(G)\geq \mw_2(G)\geq \dots \geq \mw(G)=\omega(G).$$
Since they are all lower bounded by the clique number of the graph, they are alleviated by the disadvantage seen in~\cite{diestel2005graph}, 
where the ``gridwidth'' of all graphs was bounded. Note that in our setting, a decomposition in a ``gridlike'' fashion would only be a $2$-median decomposition.
Moreover, by considering complete multipartite graphs, we establish that this infinite hierarchy of parameters is proper in the strong sense that each of its levels is ``unbounded'' in the previous ones:
for $i<i'$, graphs classes of bounded $i'$-medianwidth can have unbounded $i$-medianwidth. This also provides a natural way to go beyond treewidth and obtain new ``bounded width'' 
hierarchies of the class of all graphs, now in terms of bounded $i$-medianwidth, for different $i\geq 1$. 
Lastly, the main result of the section is a characterisation of $i$-medianwidth in terms of tree decompositions: we prove that it corresponds to the largest ``intersection'' of the best choice of
$i$ many tree decompositions of the graph.

In Section \ref{sec:gated}, we discuss a general framework of how to decompose a graph $G$ in any fixed graphlike fashion, where the underlying graph of the decomposition is chosen from an arbitrary 
fixed graph class $\HHH$, and such that the most important properties of tree and median decompositions are preserved.

Finally, in Section \ref{sec:remarks} we motivate some of the various questions that arise from the study of the concept of median decompositions.

\section{Preliminaries and Median Graphs}\label{sec:prelim}

Our notation from graph theory is standard, we defer the reader to \cite{Diestel05} for the background. For a detailed view on median graphs, the reader can refer to books
\cite{feder1995stable,imrich2000product,van1993theory} and papers \cite{bandelt1983median,klavzar1999median}, or a general survey on metric graph theory and geometry \cite{bandelt2008metric}.
In this paper, every graph we consider will be finite, undirected and simple.

For $u,v \in V(G)$, a \emph{$(u,v)$-geodesic} is a shortest $(u,v)$-path. A path $P$ in $G$ is a \emph{geodesic} if there are vertices $u,v$ such that $P$ is a $(u,v)$-geodesic. 

The \emph{interval} $I(u,v)$ consists of all vertices lying on a $(u,v)$-geodesic, namely
$$I(u,v)=\{x\in V(G) \mid d(u,v)=d(u,x)+d(x,v)\}.$$ 

A graph $G$ is called \emph{median} if it is connected and 
for any three vertices $u,v,w \in V(G)$ there is a unique vertex $x$, called the \emph{median} of $u,v,w$, 
that lies simultaneously on a $(u,v)$-geodesic, $(v,w)$-geodesic and a $(w,u)$-geodesic. In other words, $G$ is median if $|I(u,v) \cap I(v,w) \cap I(w,u)|=1$, for every three vertices $u,v,w$.

A set $S\subseteq V(G)$ is called \emph{geodesically convex} or just \emph{convex} if for every $u,v\in S$, $I(u,v)\subseteq S$
(we will only talk about geodesic convexity and not other graph convexities, 
so it is safe to refer to geodesically convex sets as just convex, without confusion).
By definition, convex sets are connected. As with convex sets in Euclidean spaces (or more generally, as a prerequisite of abstract convexities), 
it is easy to see that the intersection of convex sets is again convex. Note that the induced subgraphs corresponding to convex sets of median graphs are also median graphs.

For $S\subseteq V(G)$, its \emph{convex hull} $<S>$ is the minimum convex set of $G$ containing $S$.

For the rest of the section, we present without proofs some well-known basic theory on median graphs 
and summarize some of their most important properties, that will be important for our needs throughout the paper.

Let us fast present some examples. Let $C_k$ be the cycle graph on $k$ vertices. Notice that the cycles $C_3$ and $C_k$, where $k\geq 5$,  are not median, simply because there are always 3 vertices with no median. 
As we will later see, every median graph is bipartite. On the other hand, 
apart from the even cycles of length at least six, examples of bipartite graphs that aren't median are the complete bipartite graphs $K_{n,m}$ with $n\geq 2$ and $m \geq 3$, 
since all $n$ vertices of one part are medians of every three vertices of the other part.

The \emph{$i$-dimensional hypercube} or \emph{$i$-cube} $Q_i$, $i\geq 1$, is the graph with vertex set $\{0, 1\}^i$,
two vertices being adjacent if the corresponding tuples differ in precisely one position.
They are also the only regular median graphs \cite{mulder1980n}.

The Cartesian product $G\Box H$ of graphs $G$ and $H$ is the graph with vertex set $V(G)\times V(H)$,
in which vertices $(a,x)$ and $(b,y)$ are adjacent whenever $ab\in E(G)$ and $x=y$, or $a=b$
and $xy\in E(H)$. The Cartesian product is associative and commutative with $K_1$ as its unit.
Note that the Cartesian product of $n$-copies of $K_2=Q_1$ is an equivalent definition of the $i$-cube $Q_i$.

In the Cartesian products of median graphs, medians of vertices can be seen to correspond to the tuple of the medians in every factor of the product. 
The following Lemma is folklore.

\begin{lemma}\label{cartmed}
 Let $G=\Box_{i=1}^kG_i$, where $G_i$ is median for every $i=1,\ldots,k$. Then $G$ is also median, whose convex sets are precicely the sets $C=\Box_{i=1}^kC_i$,
 where $C_i$ is a convex subset of $G_i$.
\end{lemma}

There are several characterizations of median graphs: they are exactly the retracts of hypercubes; they can be obtained by successive 
applications of convex amalgamations of proper median subgraphs; they can also be obtained by $K_1$ after a sequence of \emph{convex} or \emph{peripheral expansions}.

A graph $G$ is a \emph{convex amalgam} of
two graphs $G_1$ and $G_2$ (along $G_1\cap G_2$) if $G_1$ and $G_2$ constitute two intersecting induced convex subgraphs of
$G$ whose union is all of $G$.

A (necessarily induced) subgraph $H$ of a graph $G$ is a \emph{retract} of $G$, if there is a map $r:V(G)\rightarrow V(H)$ that maps each edge of $G$ to an edge of $H$,
and fixes $H$, \ie, $r(v) =v$ for every $v\in V(H)$. A \emph{core} is a graph which does not retract to a proper subgraph. 
Any graph is homomorphically equivalent to a unique core. Median graphs are easily seen to be closed under retraction, 
and since they include the $i$-cubes, every retract of a hypercube is a median graph. Actually, the inverse is also true, one of whose corollaries 
is that median graphs are bipartite graphs.

\begin{theorem}\cite{bandelt1984retracts,isbell1980median,van1983matching}\label{retract}
 A graph $G$ is median if and only if it is the retract of a hypercube. Every median graph with more than two vertices is either a Cartesian
product or a convex amalgam of proper median subgraphs.
\end{theorem}

A graph $H$ is \emph{isometrically embeddable} into a graph $G$ if there is a mapping $\phi:V(H)\rightarrow V(G)$ such that $d_G(\phi(u), \phi(v))=d_H(u, v)$ for any vertices $u,v \in H$.
Isometric subgraphs of hypercubes are called \emph{partial cubes}. Retracts of graphs are isometric subgraphs, hence median graphs are partial cubes,
but not every partial cube is a median graph: $C_6$ is an isometric subgraph of $Q_3$, but not a median graph.

A pair $(A,B)$ is a \emph{separation} of $G$ if $A\cup B=V(G)$ and $G$ has no edge between $A\setminus B$ and $B\setminus A$.
Suppose that $(A,B)$ is a separation of $G$, where $A\cap B\neq \emptyset$ and $G[A], G[B]$ are isometric subgraphs of $G$. 
An \emph{expansion} of $G$ with respect to $(A,B)$ is a graph $H$ obtained from $G$ by the following steps:

\begin{enumerate}[(i)]
 \item Replace each $v\in A\cap B$ by vertices $v_1,v_2$ and insert the edge $v_1v_2$.
 \item Insert edges between $v_1$ and all neighbours of $v$ in $A\setminus B$.
       Insert edges between $v_2$ and all neighbours of $v$ in $B\setminus A$.
 \item Insert the edges $v_1u_1$ and $v_2u_2$ if $v,u \in A\cap B$ and $vu \in E(G)$.
\end{enumerate}

An expansion is \emph{convex} if $A\cap B$ is convex in $G$. We can now state Mulder's Convex Expansion Theorem on median graphs.

\begin{theorem}\cite{mulder1978structure,mulder1980interval}
 A graph is median if and only if it can be obtained from $K_1$ by a sequence of convex expansions.
\end{theorem}

For a connected graph and an edge $ab$ of $G$ we denote
\begin{itemize} 
 \item $W_{ab}=\{v\in V(G) \mid d(v,a)<d(v,b)\},$
 \item $U_{ab}=W_{ab}\cap N_G(W_{ba}).$
\end{itemize}

\noindent
Sets of the graph that are $W_{ab}$ for some edge $ab$ will be called \emph{$W$-sets} and similarly we define \emph{$U$-sets}. If $U_{ab}=W_{ab}$ for some edge $ab$, 
we call the set $U_{ab}$ a peripheral set of the graph.
Note that if $G$ is a bipartite graph, then $V(G)=W_{ab}\cup W_{ba}$ and $W_{ab}\cap W_{ba}=\emptyset$ is true for any edge $ab$.  
If $G$ is a median graph, it is easy to see that $W$-sets and $U$-sets are convex sets of $G$. 
Moreover, the $W$-sets of $G$ play a similar role to that of the halfspaces of the Euclidean spaces, which is highlighted by the following lemma:

\begin{lemma}\label{conv}
 For a median graph, every convex set is an intersection of $W$-sets.
\end{lemma}

Edges $e=xy$ and $f=uv$ of a graph $G$ are in the \emph{Djokovic-Winkler} relation $\Theta$ \cite{djokovic1973distance,winkler1984isometric} if $d_G(x,u)+d_G(y,v)\neq d_G(x,v)+d_G(y,u)$. 
Relation $\Theta$ is reflexive and symmetric. If $G$ is bipartite, then $\Theta$ can be defined as follows: $e=xy$ and $f=uv$ are in relation $\Theta$ if $d(x,u)=d(y,v)$ and $d(x,v)=d(y,u)$. 
Winkler \cite{winkler1984isometric} proved that on bipartite graphs relation $\Theta$ is transitive if and only if it is a partial cube and so, 
by Theorem \ref{retract} it is an equivalence relation on the edge set of every median graph, whose classes we call $\Theta$-classes. 

The following lemma summarizes some properties of the $\Theta$-classes of a median graph:

\begin{lemma}\cite{imrich2000product}\label{theta}
 Let $G$ be a median graph and for an edge $ab$, let $F_{ab}=F_{ba}$ denote the set of edges between $W_{ab}$ and $W_{ba}$. Then the following are true:
 \begin{enumerate}
  \item $F_{ab}$ is a matching of $G$.
  \item $F_{ab}$ is a minimal cut of $G$.
  \item A set $F\subseteq E(G)$ is a $\Theta$-class of $G$ if and only if $F=F_{ab}$ for some edge $ab\in E(G)$. 
 \end{enumerate}

\end{lemma}

An expansion with respect to a separation $(A,B)$ of $G$ is called \emph{peripheral}, if $A\subseteq B$ and $A= A\cap B$ is a convex set of $G$. In other words, if $A$ is a convex set,
the peripheral expansion along $A$ is the graph $H$ obtained by taking the disjoint union of a copy of $G$ and $A$ and joining each vertex in the copy of $A$ to its corresponding vertex of the subgraph $A$ of $G$
in the copy of $G$. Note that in the new graph $H$, the new copy of $A$ is a peripheral set of $H$, hence the name of the expansion. Moreover, during a peripheral expansion of a median graph, 
exactly one new $\Theta$-class appears. Peripheral expansions are enough to get all median graphs.

\begin{theorem}\cite{mulder1990expansion}
 A graph G is a median graph if and only if it can be obtained from $K_1$ by a sequence of peripheral expansions.
\end{theorem}

Finally, a family of sets $\FFF$ on a universe $U$ has the \emph{Helly property}, if every finite subfamily of $\FFF$ with pairwise-intersecting sets, has a non-empty total intersection.
A crucial property for our purposes is the following well-known lemma for the convex sets of a median graph.

\begin{lemma}\cite{imrich2000product}\label{helly}
The convex sets of a median graph $G$ have the Helly property. 
\end{lemma}

\section {Median Decompositions and Medianwidth of Graphs}\label{sec:meddec}

A \emph{tree decomposition} $\DDD$ of a graph $G$ is a pair $(T,\ZZZ)$, where
$T$ is a tree and $\ZZZ=(Z_t)_{t\in V(T)}$ is a family of subsets of $V(G)$
(called bags) such that 
\begin{enumerate}[(i)]
 \item[(T1)] for every edge $uv\in E(G)$ there exists $t \in V(T)$ with $u,v \in Z_t$, 
 \item[(T2)] for every $v\in V(G)$, the set $Z^{-1}(v):=\{t\in V(T) \mid v\in Z_t\}$ is a non-empty connected subgraph (a subtree) of $T$.
\end{enumerate} 
The \emph{width} of a tree decomposition $\DDD=(T,\ZZZ)$ is the number $$\max\{|Z_t|-1 \mid t\in V(T)\}.$$
The \emph{adhesion} of $\DDD$ is the number $$\max\{|Z_t\cap Z_{t'}| \mid tt'\in E(T)\}.$$
Let $\TTT^G$ be the set of all tree decompositions of $G$. The \emph{treewidth} $\tw(G)$ of $G$ is the least width of any tree decomposition of $G$, namely
$$\tw(G):=\min\limits_{\DDD\in \TTT^G}\max\{|Z_t|-1 \mid t\in V(T)\}.$$ 
One can easily check that trees themselves have treewidth $1$.

We assume familiarity with the basic theory of tree decompositions as in \cite{Diestel05} or \cite{reed1997tree}. Let the clique number $\omega(G)$ be the size of the 
largest complete subgraph of $G$. A \emph{(proper) vertex colouring} of a graph $G$ with $k$ colours is a map $c:V(G)\rightarrow \{1,\ldots,k\}$ such that $c(v)\neq c(u)$ whenever $uv\in E(G)$. 
The \emph{chromatic number} $\chi(G)$ is the smallest integer $k$ such that $G$ can be coloured with $k$ colours. 
A graph with $\chi(G)=k$ is called \emph{$k$-chromatic}, while if $\chi(G)\leq k$, we call $G$ \emph{$k$-colourable}.

We say that $H$ is a minor of $G$ and write $H\preceq_m G$, if $H$ can be obtained from $G$ by deleting edges and vertices, and by contracting edges.
In the next lemma, we summarize some of the most important well-known properties of tree decompositions.

\begin{lemma}\label{twprop}
 Let $\DDD=(T,\ZZZ) \in \TTT^G$.
 \begin{enumerate}[(i)]
  \item For every $H\subseteq G$, the pair $(T, (Z_t\cap V(H))_{t\in T})$ is a tree decomposition of $H$, so that $\tw(H)\leq \tw(G)$.
  \item Any complete subgraph of $G$ is contained in some bag of $\DDD$, hence $\omega(G)\leq \tw(G)+1$.
  \item For every edge $t_1t_2$ of $T$, $Z_{t_1}\cap Z_{t_2}$ separates $W_1:=\bigcup_{t\in T_1} Z_t$ from $W_2:=\bigcup_{t\in T_2} Z_t$,
  where $T_1,T_2$ are the components of $T-t_1t_2$, with $t_1 \in T_1$ and $t_2 \in T_2$. 
  \item If $H\preceq_m G$, then $\tw(H)\leq \tw(G)$.
  \item $\chi(G)\leq \tw(G)+1.$
 \end{enumerate}

\end{lemma}

In a tree decomposition, every vertex of the graph lives in a connected subtree of the tree. Recall that trees are median graphs.
As we already foreshadowed in Section \ref{sec:intro}, the crucial observation, which (together with the Helly property of the convex sets of median graphs) 
is actually the reason that enables the development of the whole theory in this paper,
is the following: 

\begin{center} 
\emph{A subgraph of a tree is convex if and only if it is connected.} 
\end{center}

\noindent
Inspired by this observation and the general theory on tree decompositions, it is only natural to define this concept 
of decomposition of a graph, not only on trees such that every vertex of the graph lives in a connected subtree, 
but generally on median graphs such that every vertex lives in a convex subgraph of the median graph.

A \emph{median decomposition} $\DDD$ of a graph $G$ is a pair $(M,\XXX)$, where
$M$ is a median graph and $\XXX=(X_a)_{a\in V(M)}$ is a family of subsets of $V(G)$
(called bags) such that 
\begin{itemize}
 \item[(M1)] for every edge $uv\in E(G)$ there exists $a \in V(M)$ with $u,v \in X_a$, 
 \item[(M2)] for every $v\in V(G)$, the set $X^{-1}(v):=\{a\in V(M) \mid v\in X_a\}$ is a non-empty convex subgraph of $M$.
\end{itemize} 
The \emph{width} of a median decomposition $\DDD=(T,\XXX)$ is the number $$\max\{|X_a| \mid a\in V(M)\}.\footnotemark\footnotetext{While the definition of the width of tree decompositions 
is adjusted so that trees are exactly the graphs of treewidth $1$, 
by Theorem \ref{omega=mw} all trianglefree graphs have minimum medianwidth. Since there wouldn't be a similar exact correspondence of graphs of minimum medianwidth to 
the underlying graph class of median decompositions as in the case of treewidth, we felt that such an adjustment is not meaningful for medianwidth.} $$ 

Let $\MMM^G$ be the set of all median decompositions of $G$. The \emph{medianwidth} $\mw(G)$ of $G$ is the least width of any median decomposition of $G$:
$$\mw(G):=\min\limits_{\DDD\in \MMM^G}\max\{|X_a| \mid a\in V(M)\}.$$ 

Since $\TTT^G \subseteq \MMM^G$, by definition of $\mw(G)$ we have $\mw(G)\leq tw(G)+1$. Let us find out which of the properties  of tree decompositions in Lemma \ref{twprop}
can be translated in any sense to properties of median decompositions. For the Lemmata that follow, $\DDD=(T,\XXX) \in \MMM^G$ is a median decomposition of a graph $G$. 
It is straightforward that median decompositions are passed on to subgraphs.

\begin{lemma}\label{submw}
 For every $H\subseteq G$, $(M, (X_a\cap V(H))_{a\in M})$ is a median decomposition of $H$, hence $\mw(H)\leq \mw(G)$.\qed
\end{lemma}

The Helly property of the convex sets of median graphs was the secondary reason that indicated that median decompositions seem to be a natural notion. 
It is what allows us to prove the direct analogue of Lemma \ref{twprop} (ii) (with an actual proof this time around).

\begin{lemma}\label{omega<mw}
 Any complete subgraph of $G$ is contained in some bag of $\DDD$. In particular, $\omega(G)\leq \mw(G)$.
\end{lemma}

\begin{proof}
 Let $K$ be a complete subgraph of $G$. By (M1), for every $u,v\in V(K)$, there exists a bag of $M$ that contains both $u$ and $v$, so that $X^{-1}(u)\cap X^{-1}(v)\neq \emptyset$.
 By (M2), the family $\FFF=\{X^{-1}(v) \mid v\in V(K)\}$ is a family of pairwise-intersecting convex sets of the median graph $M$. By Lemma \ref{helly}, 
 $$\bigcap\FFF=\bigcap_{v\in V(K)}X^{-1}(v)\neq \emptyset$$ and hence, there is a bag of $M$ that contains all vertices of $K$.
\end{proof}

For a median decomposition $(M,\XXX)$ and a minimal cut $F\subseteq E(M)$ of $M$ that separates $V(M)$ into $W_1$ and $W_2$, let $U_i$ be the vertices of $W_i$ adjacent to edges of $F$,
and let $Y_i:= \bigcup_{x\in W_i}X_x$, $Z_i:= \bigcup_{x\in U_i}X_x$, where $i=1,2$. 
Observe that minimal cuts on a tree are just single edges by themselves.
This leads us to an analogue of Lemma \ref{twprop}(iii), which says that minimal cuts of $M$ correspond to separations of $G$.

\begin{lemma}\label{mincut}
 For every minimal cut $F$ of $M$ and $Y_i, Z_i$, $i=1,2$, defined as above, $Z_1\cap Z_2$ separates $Y_1$ from $Y_2$.  
\end{lemma}

\begin{proof}
 Let $v \in Y_1\cap Y_2$. Then there are $a\in W_1, b \in W_2$, such that $v\in X_{a}\cap X_{b}$, \ie $a,b\in X^{-1}(v)$. 
 By the convexity of $X^{-1}(v)$, it must be $I(a,b)\subseteq X^{-1}(v)$. 
 But $F$ is a minimal cut between $W_1$ and $W_2$, therefore there is an $xy \in F$ with $x\in W_1,y\in W_2$, such that $x,y \in X^{-1}(v)$, so that $v\in X_x\cap X_y\subseteq Z_1\cap Z_2$. 
 This proves that $Y_1\cap Y_2 \subseteq Z_1\cap Z_2$. 
 
 It remains to show, that there is no edge $u_1u_2$ of $G$ with $u_1 \in Y_1\setminus Y_2$ and $u_2 \in Y_2\setminus Y_1$. 
 If $u_1u_2$ was such an edge, then by (M1) there is an $x\in V(M)$ with $u_1,u_2 \in X_x$, 
 hence $x \in X^{-1}(u_1)\cap X^{-1}(u_2)\subseteq (Y_1\setminus Y_2) \cap (Y_2\setminus Y_1)=\emptyset,$ a contradiction.
\end{proof}

Recall that by Lemma \ref{theta}, for an edge $ab$ of $M$, the $\Theta$-class $F_{ab}$ is a minimal cut of $M$. Denote $Y_{ab}:= \bigcup_{x\in W_{ab}}X_x$ and $Z_{ab}:= \bigcup_{x\in U_{ab}}X_x$. 
We will refer to them as the \emph{$Y$-sets} and \emph{$Z$-sets} of a median decomposition $\DDD$.
Note that the $Y$-sets and $Z$-sets are subsets of the decomposed graph $G$, while the $W$-sets and $U$-sets are subsets of the median graph $M$ of the decomposition.
Observe that a more special way to look at the edges of a tree is that each edge of a tree forms a degenerated $\Theta$-class by itself and its two corresponding $U$-sets are the ends of the edge. 
As a special case of Lemma \ref{mincut}, we obtain a more specific analogue of Lemma \ref{twprop}(iii),
which says that intersections of unions of bags across opposite sides of a whole $\Theta$-class of $M$ also correspond to separations of $G$.

\begin{lemma}\label{sep}
 For every edge $ab$ of $M$, $Z_{ab}\cap Z_{ba}$ separates $Y_{ab}$ from $Y_{ba}$.\qed
\end{lemma}

%%%%%%%%%%%%%%%%%%%%% Fig %%%%%%%%%%%%%%%%%%%%%%%

\begin{figure}[t]
 \centering

\begin{tikzpicture}[
  dot/.style={draw,fill=black,circle,inner sep=0.5pt}
  ]

 \foreach \i in {1,...,4} {
    \node[dot,xshift=-100,yshift=-10,label=\i*90:\tiny{$\i$}] (\i) at (\i*90:1) {};
}

\draw (1) -- (2);
\draw (2) -- (3);
\draw (3) -- (4);
\draw (4) -- (1);
 
\node[dot,fill=white] (12) at (0,0) {12};
\node[dot,fill=white] (14) at (1,0) {14};
\node[dot,fill=white] (34) at (1,-1) {34};
\node[dot,fill=white] (23) at (0,-1) {23};

\draw (12) -- (23);
\draw (23) -- (34);
\draw (34) -- (14);
\draw (14) -- (12);

\end{tikzpicture}
\caption{A median decomposition of $C_4$ of width $2$.}
\label{fig:c4}
\end{figure}
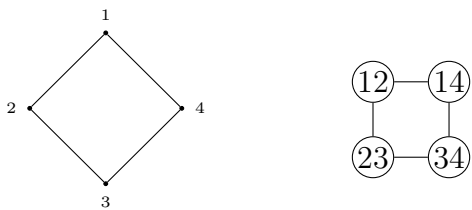

%%%%%%%%%%%%%%%%%%%%%%%%%%%%%%%%%%%%%%5

While the first three properties of Lemma \ref{twprop} can be translated into the setting of median decompositions, it is not the case that $\mw(H)\leq \mw(G)$, whenever $H \preceq_m G$.
The median decomposition of $C_4$ in Fig.~\ref{fig:c4}, shows that $\mw(C_4)\leq 2$, while (by Lemma~\ref{omega<mw}) $\mw(C_3)\geq \omega(C_3)=3$ and $C_3\preceq_m C_4$. 
An insight to why medianwidth is not a minor-closed parameter, is that 
while the union of two intersecting connected subsets of a tree is again a connected subset 
(which allows you to safely replace in the bags of a tree decomposition both vertices of a contracted edge of the original graph with the new vertex obtained by the contraction without hurting (T2)
and get a tree decomposition of the contracted graph with at most the same width), it is not true in general that the union of two intersecting convex sets of a median graph is again convex. 

The \emph{simplex graph} $\kappa(G)$ of $G$, is the graph with vertex set the set of complete subgraphs of $G$, 
where two vertices of $\kappa(G)$ are adjacent if the corresponding cliques differ by exactly one vertex of $G$. It is well-known that $\kappa(G)$ is a median graph \cite{bandelt1989embedding,barthelemy1986use}.

We have seen that $\omega(G)\leq \mw(G)\leq \tw(G)+1$ and that medianwidth isn't a minor-closed parameter. 
It is natural to ask if medianwidth is related to other non-minorclosed graph parameters between the clique number and the treewidth. In general, $\mw(G)<\tw(G)+1$,
so one immediate candidate is the clique number itself. By Lemma \ref{twprop}(v), and for reasons that will become apparent in Section \ref{sec:chi}, 
the chromatic number $\chi(G)$ is the other candidate that we thought of. $C_5$ and Fig.~\ref{fig:c5} show that the medianwidth and the chromatic number are not equivalent,
but in Section \ref{sec:chi} we will still attempt to compare the two parameters. 

As indicated by the simplex graph, it turns out that clique number is indeed the correct answer. 
While one might be able to argue by considering $\kappa(G)$, we will adopt a different approach for the proof,
which we believe that highlights that the directions we consider in Section \ref{sec:i-mw} are natural for the development of this theory.

%%%%%%%%%%%%%%%%%%%% Fig %%%%%%%%%%%%%%%%%%%%%%%%%

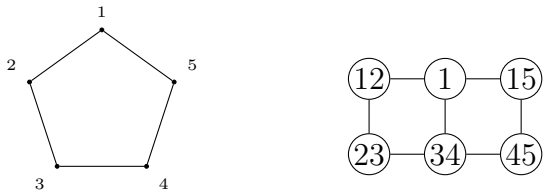
\begin{figure}[t]
 \centering
 
\begin{tikzpicture}[
  dot/.style={draw,fill=black,circle,inner sep=0.5pt}
  ]

 \foreach \i in {1,...,5} {
    \node[dot,xshift=-100,yshift=-10,label=18+\i*72:\tiny{$\i$}] (\i) at (18+\i*72:1) {};
}

\draw (1) -- (2);
\draw (2) -- (3);
\draw (3) -- (4);
\draw (4) -- (5);
\draw (5) -- (1);
 
\node[dot,fill=white] (12) at (0,0) {12};
\node[draw,circle,inner sep=2pt,fill=white] (11) at (1,0) {1};
\node[dot,fill=white] (15) at (2,0) {15};
\node[dot,fill=white] (45) at (2,-1) {45};
\node[dot,fill=white] (34) at (1,-1) {34};
\node[dot,fill=white] (23) at (0,-1) {23};

\draw (12) -- (11);
\draw (11) -- (15);
\draw (15) -- (45);
\draw (45) -- (34);
\draw (34) -- (23);
\draw (12) -- (23);
\draw (11) -- (34);

\end{tikzpicture}
\caption{$\mw(C_5)=2$, while $\chi(C_5)=3$.}
\label{fig:c5}
\end{figure}

%%%%%%%%%%%%%%%%%%%%%%%%%%%%%%%%%%%%%%%%%%%%%%%%%%%%%%%%%%%%

\begin{theorem}\label{omega=mw}
 For any graph $G$, $\mw(G)=\omega(G).$
\end{theorem}

\begin{proof}
 By Lemma \ref{omega<mw}, it is enough to show $\mw(G)\leq \omega(G)$. For a median decomposition $\DDD=(M,\XXX)$, let $\beta(\DDD)$ be the number of non-edges of $G$ contained in a bag of $\DDD$,
 namely $$\beta(\DDD):=\big|\big\{\{v,u\}\mid uv\notin E(G) \text{ and } X^{-1}(v)\cap X^{-1}(u)\neq \emptyset\big\}\big|.$$ 
 Let $\DDD_0=(M,\XXX)\in \MMM^G$ with $\beta(\DDD_0)$ minimum.
 We will prove that $\beta(\DDD_0)=0$ and therefore, every bag of $\DDD_0$ will induce a clique in $G$. Then by Lemma~\ref{omega<mw} the Theorem will follow.
 
 Suppose that $\beta(\DDD_0)>0$. Then there exists a node $a_0\in V(M)$ and two vertices in $v,u\in X_{a_0}$, such that $vu\notin E(G)$. Consider the decomposition $\DDD'=(M',\XXX')$ of $G$, 
 where:
 \begin{itemize}
  \item $M'=M\Box K_2$ is the median graph obtained by the peripheral expansion of $M$ on itself, where $V(M')=M_1\cup M_2$ and $M_1, M_2$ induce isomorphic copies of $M$. 
  Let $a_1,a_2$, be the copies of $a\in V(M)$ in $M_1,M_2$ respectively.
  \item For every $a\in V(M)$, $X'_{a_1}:=X_a\setminus \{v\}$, $X'_{a_2}:=X_a\setminus \{u\}.$
 \end{itemize}
It is straightforward to check that $\DDD'$ is a valid median decomposition of $G$, where every bag of $\DDD_0$ has been duplicated, but $u$ lives only in $M_1$ and $v$ only in $M_2$. 
Clearly, in $\DDD'$ we have that $X'^{-1}(v)\cap X'^{-1}(u)=\emptyset$, hence $\beta(\DDD')=\beta(\DDD_0)-1$, a contradiction.
\end{proof}

\section{Medianwidth vs Chromatic Number}\label{sec:chi}

As we discussed in the previous section, the chromatic number was another promising candidate, which we thought we could compare with medianwidth.
Even though the standard medianwidth is equivalent to the clique number of a graph, 
the following construction gives us an indication that suitable variations of medianwidth can become equivalent to the chromatic number. 

A $k$-dimensional lattice graph $L$ is a graph obtained by the Cartesian Product of $k$ paths. By Lemma \ref{cartmed}, lattice graphs are median graphs. 
For a $k$-colourable graph $G$, let $c:V(G)\rightarrow \{1,\ldots,k\}$ be a proper colouring of $G$ and for $i=1,\ldots,k$, let $P_i$ be a path with $|c^{-1}(i)|$ many vertices, 
whose vertices are labeled by the vertices of $c^{-1}(i)$ with arbitrary order. Consider the $k$-dimensional lattice graph $L=\Box_{i=1}^{k}P_i$, 
whose vertices $\bold{a}=(v_1,\ldots,v_k)\in V(L)$ are labeled by the $k$-tuple of labels of $v_1,\ldots,v_k$. 
For a vertex $\bold{a}\in V(L)$, define $X_\bold{a}$ to be the set of vertices that constitute the $k$-tuple of labels of $\bold{a}$. Let $\XXX=(X_\bold{a})_{\bold{a}\in V(L)}$.

\begin{lemma}\label{chromdec}
 The pair $\DDD=(L,\XXX)$ is a median decomposition of $G$ of width $k$.
\end{lemma}

\begin{proof} 
 Since every colour class $c^{-1}(i)$ is an independent set and since the bags of $\XXX$ are all the transversals of the colour classes, every edge of $G$ is contained in a bag, so that (M1) holds. 
 To see (M2), as $c$ defines a partition of $V(G)$, every vertex of $G$ will be a label in some $k$-tuple labeling a vertex of $L$, which means that there is a bag in $\XXX$ containing it.
 Let $v\in V(G)$ be the label of $x_v\in P_i$. Then $$X^{-1}(v)= \big(\Box_{j\neq i} P_j\big)\Box \{x_v\},$$ which, by Lemma \ref{cartmed}, is a convex subgraph of $L$. 
\end{proof}

We will refer to median decompositions obtained from a colouring of $V(G)$ as in Lemma~\ref{chromdec} as \emph{chromatic median decompositions}. 
Fig.~\ref{k34} shows a chromatic decomposition of a bipartite graph.
In an attempt to add some intuition to chromatic median decompositions (if needed), borrowing terminology from geometry and without elaborating more on this, 
in a chromatic median decomposition we make every vertex $v\in V(G)$ live in its own hyperplane of the lattice, a maximal sublattice of the lattice of codimension 1, which is of course convex.

%%%%%%%%%%%%%%%%%%%%%% Fig %%%%%%%%%%%%%%%%%%%%%%%

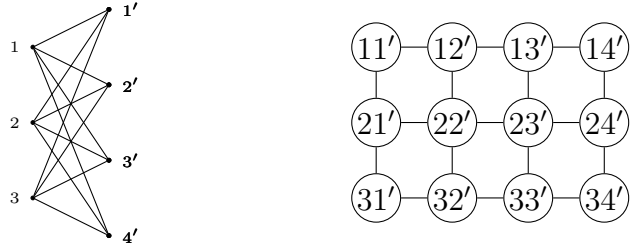
\begin{figure}
 \centering

\begin{tikzpicture}[
  dot/.style={draw,fill=black,circle,inner sep=0.5pt}
  ]

 \foreach \i in {1,2,3} {
     \node[dot,xshift=-100,yshift=-10,label=left:\tiny{$\i$}] (a\i) at (0,-\i) {};
       \foreach \j in {1,2,3,4} { 
       \node[dot,xshift=-100,yshift=-10,label=right:\tiny{$\j'$}] (b\j) at (1,0.5-\j) {};
       \draw (a\i) -- (b\j);
       
       \node[dot,yshift=-10,fill=white] (\i\j) at (\j,-\i) {$\i\j'$};

       }
       }
  
\foreach \i in {1,2,3} {
\draw (\i1) -- (\i2);
       \draw (\i2) -- (\i3);
       \draw (\i3) -- (\i4);
       }

\foreach \j in {1,2,3,4} {
   \draw (1\j) -- (2\j);
       \draw (2\j) -- (3\j);
       }

\end{tikzpicture}
\caption{A chromatic median decomposition of $K_{3,4}.$}
\label{k34}
\end{figure}

%%%%%%%%%%%%%%%%%%%%%%%%%%%%%%%%%%%%%%%%%%%%%%%%%%%%%%%%%%%%%%%%%%%%%%%%%%%%%%%%%%%%

As one can observe, chromatic median decompositions enjoy more regularity than general ones. 
One would hope that by adding in the definition of median decompositions a suitable third axiom to exploit this regularity, 
and which axiom would automatically hold for chromatic median decompositions, we would be able to make the respective variation of medianwidth equivalent to the chromatic number. 

It is well-known that a graph of treewidth~$k$ has a
tree decomposition $(T,\ZZZ)$ of width~$k$ such that for every
$st\in E(T)$ we have $|Z_s\setminus Z_t| = |Z_t\setminus Z_s| = 1$. We call such
decompositions \emph{smooth}. Recall the definition of the $Z$-sets of a median decomposition. Similarly to tree decompositions, we define a median decomposition $(M,\XXX)$ to be \emph{$\Theta$-smooth}, 
if for every $ab\in E(M)$, we have $|Z_{ab}\setminus Z_{ba}|=|Z_{ba}\setminus Z_{ab}|=1$ and additionally, $X^{-1}(v_a)\cup X^{-1}(v_b)$ is convex in $M$,
where $\{v_a\}=Z_{ab}\setminus Z_{ba}$, $\{v_b\}=Z_{ba}\setminus Z_{ab}$. 
Notice that since the $\Theta$-classes of a tree are single edges, smoothness and $\Theta$-smoothness coincide on tree decompositions.

We consider the following third axiom in the definition of median decompositions:
\begin{itemize}
 \item[](M3) $\DDD$ is $\Theta$-smooth.
\end{itemize}
The \emph{smooth-medianwidth} $\smw(G)$ of $G$ is the minimum width over all median decompositions of $G$ that additionally satisfy (M3).

\begin{lemma}\label{smooth}
 For any graph $G$, $\chi(G)\leq\smw(G)$.
\end{lemma}

\begin{proof}
 Let $\smw(G)\leq k$. Consider a $\Theta$-smooth median decomposition $(M,\XXX)$ of $G$ of width at most $k$ and $P=U_{ab}$ a peripheral set of $M$. 
 Like in the definition of $\Theta$-smoothness, let $v_a$ the single element of $Z_{ab}\setminus Z_{ba}$ and $v_b$ the single element of $Z_{ba}\setminus Z_{ab}$. 
 By Lemma \ref{sep}, $Z_{ab}\cap Z_{ba}$ separates $v_a$ and $v_b$, hence they are not adjacent in $G$. 
 Notice that since $P$ is peripheral, all neighbours of $v_a$ in $G$ are contained in $Z_{ab}\cap Z_{ba}$.
 Let $G'$ be the graph obtained by $G$ by identiying $v_a$ and $v_b$ into one new vertex $v$. 
 Then, by letting $M'=M\setminus P=M[W_{ba}]$ and replacing $v_b$ with $v$ in every bag of $X^{-1}(v_b)$ (which remains convex in $M'$), we obtain a decomposition $(M',\XXX')$ of $G'$ of width as most $k$,
 for which (M2) is immediately passed onto.
 
 To see (M1), notice that $N_{G'}(v)=N_G(v_a)\cup N_G(v_b)$. By the convexity of $X^{-1}(v_a)\cup X^{-1}(v_b)$ in $M$, 
 $X^{-1}(v_a)\cap U_{ab}$ and $X^{-1}(v_b)\cap U_{ba}$ are joined by a perfect matching $F\subseteq F_{ab}$ in $M$,
 hence $v_b$ is also contained in a common bag with every neighbour of $v_a$. This means that $v$ is contained in a common bag with everyone of its neighbours in $G'$, 
 hence $(M',\XXX')$ is a valid smooth median decomposition of $G'$. 
 
 By induction on the number of vertices of a graph with smooth-medianwidth at most $k$,
 $G'$ is $k$-colourable. Let $c'$ be a $k$-colouring of $G'$. Since $v_a$, $v_b$ are not adjacent in $G$, by letting $c(v_a)=c(v_b)=c'(v)$ and $c(u)=c'(u)$ for every $u\in V(G)\setminus \{v_a,v_b\}$,  
 we obtain a proper $k$-colouring $c$ of $G$. The Lemma follows.
\end{proof}

Lastly, by the way they are defined, chromatic median decompositions are $\Theta$-smooth, hence $\smw(G)\leq \chi(G)$. 
An immediate corollary of this observation and Lemma \ref{smooth} is the following characterization of the chromatic number.

\begin{theorem}\label{chichar}
 For any graph $G$, $\smw(G)=\chi(G).$\qed
\end{theorem}

\section{The i-Medianwidth of Graphs}\label{sec:i-mw}

For the proof of Theorem \ref{omega=mw}, we promised an approach that indicates which directions we can consider to develop this theory. 
We believe this is the case, because the proof makes apparent the fact that in order to find a median decomposition of width equal to the clique number, 
our underying median graph of the decomposition might need to contain hypercubes of arbitrarily large dimension as induced subgraphs or, 
more generally, it might need to contain Cartesian products of arbitrarily many factors. 
There are many notions of dimension for median graphs (or, more generally, partial cubes) in the literature \cite{ovchinnikov2011graphs,Eppstein2005585}.
The one most suitable for our purposes is the \emph{tree dimension} of a graph $G$, the minimum $k$ such that $G$ has an isometric embedding into a Cartesian product of $k$ trees.
The graphs with finite tree dimension are just the partial cubes \cite{ovchinnikov2011graphs}, hence every median graph has finite tree dimension.
Since trees are exactly the median graphs of tree dimension $1$, we are led to the following definition.

Fon an $i\geq 1$, an \emph{$i$-median decomposition} of $G$ is a median decomposition $\DDD=(M,\XXX)$ satisfying (M1),(M2), where $M$ is a median graph of tree dimension at most $i$.
We denote the set of $i$-median decompositions of $G$ as $\MMM^G_i$. The \emph{$i$-medianwidth} $\mw_i(G)$ of $G$ is the least width of any $i$-median decomposition of $G$:
$$\mw_i(G):=\min\limits_{\DDD\in \MMM^G_i}\max\{|X_a| \mid a\in V(M)\}.$$ 
The $1$-median decompositions are the tree decompositions of $G$, therefore $\mw_1(G)=\tw(G)+1$. By definition, the invariants $\mw_i$ form a non-increasing sequence:
$$\tw(G)+1=\mw_1(G)\geq \mw_2(G)\geq \dots \geq \mw(G)=\omega(G).$$

\noindent
An immediate observation is that $i$-medianwidth is not a bounded parameter on all graphs. 
Furthermore, we would like that $i$-medianwidth and $i'$-medianwidth for different $i,i'\geq 1$ do not constitute the same parameters, 
so that the hierarchy above is one that makes sense. In fact, we will see that complete multipartite graphs establish this in a notably strong fashion:
for $i<i'$, a class of graphs of bounded $i'$-medianwidth can have unbounded $i$-medianwidth.

For a Cartesian product of trees $H=\Box_{j=1}^kT^j$, let $\pi_j:\Box_{j=1}^kT^j\rightarrow T^j$ be the $j$-th projection of $H$ to its $j$-th factor $T^j$.
We can always embed a median graph into a Cartesian product of trees
that isn't unnecessarily large.

\begin{lemma}\label{optemb}
 Let $k$ be the tree dimension of a median graph $M$. Then there is an isometric embedding $\phi$ of $M$ 
 into the Cartesian product of $k$ trees $\Box_{j=1}^kT^j$ such that for every $j=1,\ldots,k$ and every $t^j\in E(T^j)$,
 $$\pi_j^{-1}(t^j)\cap \phi(V(M))\neq\emptyset.$$
\end{lemma}

\begin{proof}
 Let $\phi:M\rightarrow H=\Box_{j=1}^kT^j$ be an isometric embedding into the Cartesian product of $k$ trees $H$ with $V(H)$ minimal. 
 Then, for every $j=1,\ldots,k$ and every leaf $l^j\in V(T^j)$ it must be $\pi_j^{-1}(l^j)\cap \phi(V(M))\neq\emptyset$,
 otherwise we can embed $M$ into $(\Box_{h\neq j}T^h)\Box (T^j-l^j)$, a contradiction to the choice of $H$.
 Since $\phi(M)$ is a connected subgraph of $H$, the Lemma follows.
\end{proof}

We say that two $\Theta$-classes $F_{x_1x_2}, F_{x'_1x'_2}$ of a median graph $M$ \emph{cross} if $W_{x_ix_{3-i}}\cap W_{x'_jx'_{3-j}}\neq \emptyset$ for any $i,j=1,2$. 
Otherwise, if there is a choice $i,j \in \{1,2\}$ such that $W_{x_ix_{3-i}}\subseteq W_{x'_jx'_{3-j}}$ and $W_{x_{3-i}x_i}\subseteq W_{x'_{3-j}x'_j}$, we call $F_{x_1x_2}, F_{x'_1x'_2}$ \emph{laminar}.
Two $U$-sets are laminar if their adjacent $\Theta$-classes are laminar.

For a median graph $M$, let $\varTheta^M$ be the set of its $\Theta$-classes, $\UUU^M$ the family of its $U$-sets
and $\PPP^M$ the family of its peripheral sets.

A \emph{$\Theta$-system} of $M$  is a set of $\Theta$-classes of it. We call a $\Theta$-system of $M$ a \emph{direction} in $M$ if all of its members are pairwise laminar. 
In \cite{bandelt1989embedding}, Bandelt and Van De Vel show that a median graph is isometrically embeddable into the Cartesian product of $k$
trees if and only if $\Theta^M$ can be ``covered'' with $k$ directions. We will extensively use the one implication of the above result, 
which we reformulate (together with some facts obtained from its proof) in a more convenient way for what follows. For a mapping $\psi:G\rightarrow H$ and an edge $e\in E(H)$,
by $\psi^{-1}(e)$ we mean $\{uv\in E(G)\mid \psi_j(u)\psi_j(v)=e\}.$

\begin{lemma}\cite{bandelt1989embedding}\label{embedding}
 Let $\phi:M \rightarrow H$ be an isometric embedding of a median graph $M$ into the Cartesian product of $k$ trees $H=\Box_{j=1}^kT^j$ as in Lemma~\ref{optemb}. 
 Then for every $j=1,\ldots,k$ the following are true:
 \begin{enumerate}[(i)]
  \item for every $e^j\in E(T^j)$,  $\phi^{-1}(\pi_j^{-1}(e^j))$ is a $\Theta$-class of $M$
  \item the family $\Delta_j=\{\phi^{-1}(\pi_j^{-1}(e^j))\mid e^j\in E(T^j)\}$ is a direction of $M$
  \item for every node $t^j$ adjacent to an edge $e^j$ in $T^j$, one of the two $U$-sets of $M$ adjacent to 
  $\phi^{-1}(\pi_j^{-1}(e^j))$ is a subset of $\phi^{-1}(\pi_j^{-1}(t^j))$.
 \end{enumerate}
\end{lemma}

\noindent
We say that a set of vertices $S\subseteq V(G)$ \emph{intersects} a subgraph $H$ of a graph $G$ if it contains a vertex of $H$. We need the following lemma:

\begin{lemma}\cite{mulder1990expansion}\label{hull}
 Let $S$ be a set of vertices intersecting every peripheral set of a median graph $M$. Then $<S>=V(M)$.
\end{lemma}

\noindent
As promised, let us now show that complete $i+1$-partite graphs have unbounded $i$-medianwidth and thus strongly distinguish $\mw_{i+1}$ from $\mw_i$.

\begin{lemma}\label{ipartite}
 For every $i\geq 1$, $\mw_{i}(K_{n_1,\ldots,n_{i+1}})\geq \min_{j=1}^{i+1}\{n_j\}+1$, while $\mw_{i+1}(K_{n_1,\ldots,n_{i+1}})=i+1$.
\end{lemma}

\begin{proof}
 Let $K=K_{n_1,\ldots,n_{i+1}}$. Since complete $i+1$-partite graphs are $i+1$-colourable, its clique number and a chromatic median decomposition of it establish that $\mw_{i+1}(K)=i+1$.
 
 Let $(M,\XXX)$ be an $i$-median decomposition of $K$. We can assume that $|V(M)|\geq 2$ (since $K$ is not a clique) and that for every peripheral set $U_{ab}$, 
 it must be $Z_{ab}\setminus Z_{ba}\neq \emptyset$ (otherwise we just remove the peripheral set and its bags and obtain a median decomposition of $K$ with fewer bags).
 We call the vertices in $Z_{ab}\setminus Z_{ba}$ and the sets $Z_{ab}$ for some peripheral set $U_{ab}$, 
 \emph{the peripheral vertices} and \emph{the peripheral $Z$-sets} (of $K$), respectively,  with respect to the decomposition. 
 The \emph{peripheral bags} of $(M,\XXX)$ are the bags corresponding to nodes belonging to peripheral sets of $M$.
 
 Let $k\leq i$ be the tree dimension of $M$ and let $\phi:M\rightarrow \Box_{j=1}^kT^j$ be an isometric embedding into the Cartesian product of $k$ trees $H$ as in Lemma~\ref{optemb}.
 Since the peripheral sets of $H$ correspond to the leaves of the factors of $H$, it clearly follows that 
 $$\PPP^M=\{\phi^{-1}(\pi_j^{-1}(l^j))\mid j=1,\ldots,k \text{ and } l^j \text{ is a leaf of } T^j\}.$$
 We partition the peripheral sets of $M$ as inherited by the natural partition of $\PPP^H$ into the families corresponding to the leaves of each tree factor of $H$,
 namely we partition $\PPP^M$ into the sets $\PPP_1^M,\ldots,\PPP_k^M$,
 where for $j=1,\ldots,k$, $$\PPP_j^M=\{\phi^{-1}(\pi_j^{-1}(l^j))\mid l^j \text{ is a leaf of } T^j\}.$$
 
 By Lemma~\ref{embedding}, the sets of every $\PPP_j^M$ are adjacent to $\Theta$-classes which belong to the same direction. Hence,
 $\PPP_j^M$ consists of pairwise laminar peripheral sets of $M$, so, by Lemma~\ref{sep}, two peripheral vertices of $Z$-sets corresponding 
 to different peripheral sets of the same $\PPP_j^M$ are always non-adjacent in $K$. 
 It follows that every transversal of peripheral vertices chosen from different $Z$-sets corresponding to peripheral sets from the same family $\PPP_j^M$ is an independent set in $K$. 
 Recall that $|V(M)|\geq 2$, and therefore each $\PPP_j^M$ has at least two elements.
 Moreover, since $K$ is complete multipartite, if $uv,vw\notin E(K)$, then also $uw\notin E(K)$. It follows that
 all the peripheral vertices belonging to $Z$-sets corresponding to the same $\PPP_j^M$ belong to the same part of $K$, for all $j=1,\ldots,k.$
 
 But $k\leq i$ and thus, there is a part $A_{j_0}$ of $K$ that contains no peripheral vertices with respect to $(M,\XXX)$.
 As the neighbourhood of a peripheral vertex must lie completely in the corresponding $Z$-set, every vertex of $A_{j_0}$ is contained in every peripheral $Z$-set. Namely, for every vertex $v$ in $A_{j_0}$,
 $X^{-1}(v)$ intersects every peripheral set of $M$. By the convexity of $X^{-1}(v)$ and Lemma~\ref{hull},
 $v$ must belong to every bag of $(M,\XXX)$. Hence, there are peripheral bags that contain the whole $A_{j_0}$ plus a peripheral vertex of $G$, 
 so that the width of $(M,\XXX)$ is at least $|A_{j_0}|+1$. 
 As $(M,\XXX)$ was arbitrary, the lemma follows.
\end{proof}

We call two separations $(U_1,U_2), (W_1,W_2)$ of a graph $G$ \emph{laminar} if there is a choice $i,j \in \{1,2\}$ such that $U_i\subseteq W_j$ and $U_{3-i} \supseteq W_{3-j}$, 
otherwise we say they \emph{cross}. A set of separations is called laminar if all of its members are pairwise laminar separations of $G$.

\begin{lemma}\label{lammed}
 Let $(M,\XXX)$ a median decomposition of $G$.  If the $\Theta$-classes $F_{ab}$, $F_{cd}$ are laminar in $M$, 
 then the corresponding separations $(Y_{ab},Y_{ba})$ and $(Y_{cd},Y_{dc})$ are laminar in $G$.
\end{lemma}

\begin{proof}
 Let $F_{ab}$, $F_{cd}$ be laminar in $M$. Then, $F_{cd}\subseteq E(M[W_{ab}])$ or $F_{cd}\subseteq E(M[W_{ab}])$, otherwise $F_{ab}$, $F_{cd}$ cross. W.l.o.g  we can assume $F_{cd}\subseteq E(M[W_{ab}])$.
 Then $W_{cd}\subseteq W_{ab}$ and $W_{dc}\supseteq W_{ba}$. It follows that $Y_{cd}\subseteq Y_{ab}$ and $Y_{dc}\supseteq Y_{ba}$, therefore $(Y_{ab},Y_{ba})$, $(Y_{cd},Y_{dc})$ are laminar in $G$. 
\end{proof}

Note that the converse is in general not true. If $F_{ab}$, $F_{cd}$ cross in $M$, but at least one of the four sets $(Y_{ab}\setminus Y_{ba}) \cap (Y_{cd}\setminus Y_{dc})$, 
$(Y_{ab}\setminus Y_{ba}) \cap (Y_{dc}\setminus Y_{cd})$, $(Y_{ba}\setminus Y_{ab}) \cap (Y_{cd}\setminus Y_{dc})$, $(Y_{ba}\setminus Y_{ab}) \cap (Y_{dc}\setminus Y_{cd})$ is empty, 
then $(Y_{ab},Y_{ba})$, $(Y_{cd},Y_{dc})$ are still laminar in $G$. Moreover, one can see that the proof of Lemma~\ref{lammed} also works
if one defines laminarity not only for $\Theta$-classes, but for general minimal cuts of the median graph $M$ in the natural way, so
Lemma~\ref{lammed} holds for general laminar minimal cuts of $M$ accordingly.

In \cite{robertson1991graph}, Robertson and Seymour construct the so-called \emph{standard tree decomposition} of a graph into its \emph{tangles} 
(the definition of which we omit, since we don't need it for this paper).
To do that, they make use of the following lemma, also used by Carmesin \emph{et al.}~in \cite{carmesin2014connectivity} (where laminar separations stand under the name \emph{nested separations}), 
which we will also need.

\begin{lemma}\label{lamsep}
 For a tree decomposition $(T,\ZZZ)$ of $G$, the set of all separations of $G$ that correspond to the edges of  $T$ as in Lemma \ref{twprop}(iii) is laminar. Conversely,
 if $\{(A_i,B_i) \mid 1\leq i \leq k\}$ is a laminar set of separations of $G$, there is a tree decomposition $(T,\ZZZ)$ of $G$ such that 
 \begin{enumerate}[(i)]
  \item for $1\leq i \leq k$, $(A_i,B_i)$ corresponds to a unique edge of $T$
  \item for each edge $e$ of $T$, at least one of the separations of the two separations that corresponds to $e$ equals $(A_i,B_i)$ for some $i\in \{1,\ldots,k\}$.  
 \end{enumerate} 
\end{lemma}
 
We are ready to present the main result of this section, which roughly says that the $i$-medianwidth of a graph corresponds to the largest ``intersection'' of the best choice of
$i$ many tree decompositions of the graph. In the following theorem, when we denote tree decompositions with $\DDD^j$, we mean $\DDD^j=(T^j,\ZZZ^j)$.

\begin{theorem}\label{imwchar}
 For any graph $G$ and any integer $i\geq 1$, $$\mw_i(G)=\min\limits_{\DDD^1,\ldots,\DDD^i\in \TTT^G}\max\{|\bigcap_{j=1}^iZ^j_{t_j}|\mid t_j\in V(T^j)\}.$$
\end{theorem}

\begin{proof}
 Let $$\mu:=\min\limits_{\DDD^1,\ldots,\DDD^i\in \TTT^G}\max\{|\bigcap_{j=1}^iZ^j_{t_j}|\mid t_j\in V(T^j)\}.$$
 For $\DDD^1,\ldots,\DDD^i\in \TTT^G$, consider the pair $(M,\XXX)$, where $M=\Box_{j=1}^iT^j$ and $X_{(t_1,\ldots,t_i)}=\bigcap_{j=1}^iZ^j_{t_j}$.
 Observe that (M1) follows directly by (T1) for $\DDD^1,\ldots,\DDD^i$. Moreover, for every $v\in V(G)$, we have
 $$X^{-1}(v)=\Box_{j=1}^iZ^{j^{-1}}(v),$$ which, by Lemma \ref{cartmed}, is a convex subset of $M$, so (M2) also holds. 
 Then $(M,\XXX)$ is a valid $i$-median decomposition of $G$, therefore $$\mw_i(G) \leq \max\{\bigcap_{j=1}^iZ^j_{t_j}\mid t^j\in V(T^j)\}.$$
 Since $\DDD^1,\ldots,\DDD^i$ were arbitrary, it follows that $\mw_i(G)\leq \mu$. 
 
 For the opposite implication, consider an $i$-median decomposition $(M,\XXX)$ of $G$ of width $\mw_i(G)$. 
 Let $k\leq i$ be the tree dimension of $M$ and let $\phi:M\rightarrow H=\Box_{j=1}^kT^j$ be an isometric embedding as per Lemma~\ref{optemb}.
 By Lemma~\ref{embedding}(i),(ii), each $$\Delta_j=\{\phi^{-1}(\pi_j^{-1}(e^j))\mid e^j\in E(T^j)\}$$ is a direction in $M$. 
 By the definition of a direction, Lemma~\ref{lammed} and Lemma~\ref{lamsep}, there are tree decompositions $\DDD^j=(T^j,\ZZZ^j)$ of $G$
 obtained by each $\Delta_j$ and by Lemma~\ref{embedding}(iii), for each $t^j\in V(T^j)$ we have
 $$Z_{t^j}^j=\bigcup_{\pi_j(\phi(a))=t^j}X_a.$$
 Observe that for each $a\in V(M)$, it is $$\{a\}=\bigcap_{\substack{\pi_j(\phi(a))=t^j\\ j=1,\ldots,k}}\phi^{-1}(\pi_j^{-1}(t^j)).$$  
 It follows that
 $$X_a=\bigcap_{\substack{\pi_j(\phi(a))=t^j\\ j=1,\ldots,k}}Z_{t^j}^j.$$
 Clearly, the maximal intersections of bags, one taken from each of $\DDD^1,\ldots,\DDD^k$, correspond to the elements of $\XXX$.
 Therefore, by considering for $\mu$ the decompositions $\DDD^1,\ldots,\DDD^k$
 together with the trivial decomposition of $G$ consisting of one bag being the whole $V(G)$ and repeated $i-k$ times, we obtain
  $$\mu \leq \max\{|\bigcap_{j=1}^{k}Z^j_{t_j}|\mid t_j\in V(T^j)\}=\max\{|X_a|\mid a\in V(M)\}=\mw_i(G).$$  
\end{proof}

The combination of Theorems \ref{omega=mw} and \ref{imwchar}, imply the following, rather unnatural characterisation of the clique number.
 
 \begin{theorem}
 Let $\overline{m}=|E(G)^c|$. Then $$\omega(G)=\min\limits_{\DDD^1,\ldots,\DDD^{\overline{m}}\in \TTT^G}\max\{|\bigcap_{j=1}^iZ^j_{t_j}|\mid t_j\in V(T^j)\}.$$\qed
 \end{theorem}
 
Recall that for a $k$-colourable graph a corresponding chromatic median decomposition is, by Lemma \ref{chromdec}, a $k$-median decomposition of width $k$. This immediately implies the following.

\begin{lemma}\label{mwchi}
 For any graph $G$, $\mw_{\chi(G)}\leq \chi(G)$.\qed
\end{lemma}
 
Moreover, to obtain Theorem~\ref{chichar} we can clearly choose to restrict to $\Theta$-smooth median decompositions where the underlying median graph is always a Cartesian product of trees. 
In such a case, by $\Theta$-smoothness all the tree decompositions obtained following the directions in the Cartesian product as in Lemma~\ref{embedding} are smooth. 
Let $\TTT^G_{smooth}$ be the set of smooth tree decompositions of $G$.
A direct adaptation of the proof of Theorem \ref{imwchar} combined with Theorem \ref{chichar} provide an alternative (and seemingly unintuitive) characterization of the chromatic number 
with respect to smooth tree decompositions.

\begin{theorem}
 A graph $G$ is $k$-chromatic if and only if $$\min\limits_{\DDD^1,\ldots,\DDD^k\in \TTT^G_{smooth}}\max\{|\bigcap_{j=1}^kZ^j_{t_j}|\mid t_j\in V(T^j)\}=k.$$\qed
\end{theorem}

It might still be interesting to study the non-increasing sequence of the corresponding \emph{smooth $i$-medianwidth} invariants, 
starting from treewidth and converging to the chromatic number, a direction which we will not pursue in this paper.

\section{More General Decompositions}\label{sec:gated}

Let $K$ be a subset of vertices in a graph $G$, and
let $u \in V(G)$. A \emph{gate} for $u \in K$ is a vertex $x \in K$ such that $x$ lies in $I(u, w)$, for
each vertex $w \in K$. Trivially, a vertex in $K$ is its own gate. Moreover, if $u$ has a gate in $K$, then it must be unique and it is the vertex in
$K$ closest to $u$.  A subset $K$ of $V(G)$ is called \emph{gated},
if every vertex $v$ of $G$ has the gate $p_K(v)$ in $K$.

Some general properties of gated sets are
that every gated set is also geodesically convex (see \cite{dress1987gated}), that a map which maps a vertex to
its gate in a gated set is a retraction (see Lemma 16.2 in \cite{imrich2000product}), that the intersection
of two gated sets yields a gated set again (see Lemma 16.3 in \cite{imrich2000product}) and, very importantly, that the family of
gated sets has the Helly property (see Corollary 16.3 in \cite{imrich2000product}). In the case of median graphs, 
gated sets are exactly the convex sets (see Lemma 12.5 in \cite{imrich2000product}). 

Lemma~\ref{omega<mw}, which essentially says that cliques behave as a compact, inseparable object of the decomposed graph, 
can be also seen in the following way: when the decomposition is seen as a hypergraph on the vertex set 
of the decomposed graph with hyperedges the bags of the decomposition, a tree or median decomposition becomes a
\emph{conformal hypergraph}\footnotemark \ that covers the edges of the decomposed graph. 

If we want to decompose a graph modelling it after any certain kind of graphs and in a way that the most characteristic properties of tree and median decompositions are preserved, 
like the one described above, then gated sets seem to provide a natural tool for such decompositions, exactly like convex sets do for median decompositions.
\footnotetext{A hypergraph $H$ is conformal if the hyperedges of its dual hypergraph H* satisfy the Helly Property.}

Let $\HHH$ be a class of graphs. An \emph{$\HHH$-decomposition} $\DDD$ of a graph $G$ is a pair $(H,\XXX)$, where
$H\in \HHH$ and $\XXX=(X_h)_{h\in V(H)}$ is a family of subsets of $V(G)$, such that 
\begin{enumerate}[(i)]
 \item[($\HHH$1)] for every edge $uv\in E(G)$ there exists $h \in V(H)$ with $u,v \in X_h$, 
 \item[($\HHH$2)] for every $v\in V(G)$, the set $X^{-1}(v):=\{h\in V(H) \mid v\in X_h\}$ is a non-empty gated set of $H$.
\end{enumerate} 
The \emph{width} of an $\HHH$-decomposition $\DDD=(H,\XXX)$ is the number $$\max\{|X_h| \mid h\in V(H)\}.$$
The \emph{$\HHH$-width} $\hw(G)$ of $G$ is the least width of any $\HHH$-decomposition of $G$.

Since the Helly property holds for the gated sets of any graph, a direct imitation of the proof of Lemma \ref{omega<mw} shows that every clique of a graph has to be fully contained in some bag of any 
$\HHH$-decomposition, so that $\omega(G)\leq\hw(G)$ (and hence $\hw$ is an unbounded parameter when considered on all graphs).
Moreover, the convexity of gated sets ensures that the analogue of Lemma~\ref{mincut} holds for general $\HHH$-decompositions as well. 
Lastly, general laminar cuts in the decomposition graph $H$ correspond to laminar separations in the decomposed graph $G$, exactly as in Lemma~\ref{lammed}.

In the case that the structure of the gated sets of the graphs of a class $\HHH$ is relatively poor, the corresponding decompositions are not very flexible. 
For example, the gated sets of a clique are only the singletons and the whole clique itself. 
For a vertex set $S\subseteq V(G)$, let $\CCC^G(S)$ be the set of components of $G\setminus S$.
It is easy to see then that when $\KKK$ is the graph class of all cliques, the corresponding width parameter is 
$$\kw(G)=\min_{S\subseteq V(G)}\max\{|S\cup C|\mid C\in \CCC^G(S)\}.$$

On the other hand, letting $\HHH$ be the class of cliques doesn't seem to be the natural direction one would want to take, when trying to decompose a graph. 
In general, one would want to decompose a graph in a sparser graphlike structure than the graph itself, not in denser ones like the cliques, 
so in such cases a richer structure of gated sets than the trivial ones of the cliques might then be expected. 

For example, there is a wide variety of generalizations of median graphs, whose structure is closely related to gated sets. 
A bipartite generalization of median graphs are the modular graphs. Most of other generalizations of
median graphs connected with gated sets are non-bipartite. These include quasi-median graphs
\cite{bandelt1994quasi,mulder1980interval}, pseudo-median graphs \cite{bandelt1991pseudo}, weakly median graphs \cite{bandelt2000decomposition}, pre-median graphs \cite{chastand2001fiber}, 
fiber-complemented graphs \cite{chastand2001fiber}, weakly modular graphs \cite{brevsar2002natural,chepoi1989classification}, cage-amalgamation graphs \cite{brevsar2009cage},
absolute C-median graphs \cite{brevsar2002natural} and bucolic graphs \cite{brevsar2013bucolic}. 

\section{Concluding Remarks}\label{sec:remarks}

There are numerous directions worth looking into that stem from the development of this theory. We highlight some of the ones that we consider the most important.

\subsection{Brambles}

In a graph $G$, we say that two subsets of $V(G)$ \emph{touch} if they have a vertex in common or there is an edge in $G$ between them. 
A \emph{bramble} $\BBB$ is a set of mutually touching connected vertex sets of $G$. A subset of $V(G)$ is said to \emph{cover} $\BBB$ if it meets every element of $\BBB$.
The least number of vertices that cover a bramble is the \emph{order} of that bramble. We denote the set of all brambles of $G$ with $\mathscr{B}^G$.

Brambles are canonical obstructions to small treewidth, as shown by the following Theorem of \cite{seymour1993graph}, sometimes also called the \emph{treewidth duality Theorem}.

\begin{theorem}[Seymour \& Thomas]\label{bramchar}
 Let $k\geq 0$ be an integer. A graph has treewidth at least $k$ if and only if it contains a bramble of order strictly greater than $k$.
\end{theorem}

Inspired by Theorem \ref{imwchar} and its proof, one might think that brambles with large minimum intersections of covers are the corresponding obstructions to $i$-medianwidth.
Using Theorem \ref{imwchar}, it is not difficult to prove that the quantity 
$$\max\limits_{\BBB^1,\ldots,\BBB^i\in \mathscr{B}^G}\min\{|\bigcap_{j=1}^iX^j|\mid X^j \text{ covers } \BBB\}$$
is a lower bound for $\mw_i(G)$.

However, it is unknown to us if $\mw_i(G)$ can be upper-bounded by such a quantity and thus, we do not know if this is the correct obstructing notion characterizing large $i$-medianwidth.
We believe this is an important question towards a better comprehension of this theory.

\subsection{Towards the Chromatic Number}

A median decomposition $(M,\XXX)$ is called \emph{weakly-$\Theta$-smooth} if for every $\Theta$-class $F_{ab}$ of $M$,
we have that both $Z_{ab}\setminus Z_{ba}$ and $Z_{ba}\setminus Z_{ab}$ are non-empty, and whenever $|Z_{ab}|\leq |Z_{ba}|$,
there is an injective function $s_{ab}:Z_{ab}\setminus Z_{ba} \rightarrow Z_{ba}\setminus Z_{ab}$ such that:
\begin{itemize} 
 \item $X^{-1}(v)\cup X^{-1}(s_{ab}(v))$ is convex in $M$,
 \item for every $xy\in F_{ab}$ with $x\in U_{ab}$ and $y\in U_{ba}$, 
       \begin{center}
         $v \in X_x$ if and only if $s_{ab}(v) \in X_y$.
       \end{center}
\end{itemize}

As is easily seen, tree decompositions are always weakly-$\Theta$-smooth. Moreover, every $\Theta$-smooth median decomposition can be seen to be weakly-$\Theta$-smooth, 
by defining $s_{ab}$ to send the single element of $Z_{ab}\setminus Z_{ba}$ to the single element of $Z_{ba}\setminus Z_{ab}$. 

Consider the following variation of a third axiom in the definition of median decompositions:

\begin{itemize}
 \item[](M3') $\DDD$ is weakly-$\Theta$-smooth.
\end{itemize}

Let the \emph{weakly-smooth-medianwidth} $\wsmw(G)$ of $G$ to be the minimum width over all median decompositions of $G$ that additionally satisfy (M3'). 
A direct adaptation of the proof of Lemma \ref{smooth} shows that it is still the case that $\wsmw(G)=\chi(G)$. Nevertheless, 
even though \emph{weak $\Theta$-smoothness} is indeed a weaker notion than $\Theta$-smoothness, 
it does not seem to enhance substantially more our understanding of the chromatic number compared to $\Theta$-smoothness.

In the end, the third axiom ensures the following: if you add edges to a graph to make every bag of a median decomposition of it a clique, the new graph will be \emph{perfect}, 
one whose clique number and chromatic number coincide.
We believe though, that if there is a substantially better notion than smoothness that captures this intuition, it will be a much less artificial one than weak $\Theta$-smoothness.

\subsection{Algorithmic Considerations}

Even though treewidth is known to have a wide variety of algorithmic applications using dynamic programming techniques, this can in general not be the case for $i$-medianwidth when $i\geq 2$:
by Lemma \ref{mwchi}, all bipartite graphs have $2$-medianwidth at most $2$ and
most of the graph problems considered on graphs of bounded treewidth remain as hard in the bipartite case as in the general case.

However, it might still be meaningful to study \textsc{Minimum Vertex Cover}
(or \textsc{Maximum Independent Set}) on graphs of bounded $i$-medianwidth, 
which are known to be efficiently solvable on bipartite graphs.

Lastly, by \cite{bodlaender1993linear}, deciding the treewidth of a graph (which is the 1-medianwidth) is fixed-parameter tractable, while by \cite{downey1995fixed}, 
deciding the clique number (which is the infinite version of $i$-medianwidth)
is complete for the complexity class W[1]. It is unknown to us what the complexity of deciding the $i$-medianwidth of a graph is, for any fixed $i\geq 2$.

\section*{Acknowledgements}
The author would like to thank Martin Grohe for discussions on the subject, and Hans Bandelt along with the rest of the audience in Reinhard Diestel's research seminar for
noticing a gap in an earlier version of Theorem~\ref{imwchar}.

\bibliographystyle{abbrv}
\bibliography{mwArxiv} 

\end{document}